\documentclass[11pt]{amsart}
\usepackage{amssymb}
\usepackage{amsmath}
\usepackage{amsfonts}
\usepackage{graphicx}

\usepackage[total={17cm,22cm},top=2.5cm, left=2.3cm]{geometry}
\usepackage{hyperref}
    \usepackage{aeguill}
    \usepackage{type1cm}

\usepackage[shortlabels]{enumitem}

\newtheorem{thm}{Theorem}

\newtheorem{lem}[thm]{Lemma}

\newtheorem{prop}[thm]{Proposition}
\newtheorem{rem}[thm]{Remark}
\newtheorem{ex}[thm]{Example}

\newtheorem{defn}[thm]{Definition}

\usepackage{amsxtra}
\usepackage{eucal}
\usepackage{mathrsfs}
\usepackage{longtable}
\usepackage{caption}

\usepackage{hyperref}
    \usepackage{aeguill}
    \usepackage{type1cm}

\newcommand{\N}{\mathbb{N}}

\newcommand{\R}{\mathbb{R}}

\usepackage[usenames,dvipsnames]{color}

\usepackage[dvipsnames]{xcolor}

\begin{document}

\title{On the global shape of continuous convex functions on Banach spaces}

\author{Daniel Azagra}
\address{ICMAT (CSIC-UAM-UC3-UCM), Departamento de An{\'a}lisis Matem{\'a}tico y Matem\'atica Aplicada,
Facultad Ciencias Matem{\'a}ticas, Universidad Complutense, 28040, Madrid, Spain.  
}
\email{azagra@mat.ucm.es}
%
%
%

\date{October 27, 2019}



\begin{abstract}
We make some remarks on the global shape of continuous convex functions defined on a Banach space $Z$. For instance we show that if $Z$ is separable then for every continuous convex function $f:Z\to\R$ there exists a unique closed linear subspace $Y_f$ of $Z$ such that, for the quotient space $X_f :=Z/Y_{f}$ and the natural projection $\pi:Z\to X_f$, the function $f$ can be written in the form
$$
f(z)=\varphi(\pi(z)) +\ell(z) \textrm{ for all } z\in Z,
$$
where $\ell\in X^{*}$ and $\varphi:X_f\to\R$ is a convex function such that $\lim_{t\to\infty}\varphi(x+tv)=\infty$ for every $x, v\in X_f$ with $v\neq 0$. This kind of result is generally false if $Z$ is nonseparable (even in the Hilbertian case $Z=\ell_{2}(\Gamma)$ with $\Gamma$ an uncountable set).

\end{abstract}

\maketitle

\section{Introduction and main results}

If a continuous convex function $f$ defined on a real Banach space $(Z, \|\cdot\|)$ is coercive (meaning that $\lim_{\|z\|\to \infty}f(z)=\infty$) then $f$ will be easier to analyze and will have some nice properties that may become very useful or even crucial in some applications (an obvious example is the existence of minimizers in the case that $Z$ is reflexive). A strongly related kind of nice convex functions is what one can call {\em essentially  coercive convex functions}, namely convex functions which are coercive up to linear perturbation. Of course, even in the case $Z=\R^n$, not every convex function is essentially coercive, but it is nonetheless true (see \cite[Lemma 4.2]{Azagra2013} and \cite[Theorem 1.11]{AzagraMudarra2}) that every convex function $f:\R^n\to\R$ admits a decomposition of the form
$$
f=g\circ P+\ell,
$$
where $P$ is the orthogonal projection onto some subspace $X$ of $\R^n$ (possibly $\{0\}$ or $\R^n$), $g:X\to\R$ is convex and coercive, and $\ell:\R^n\to\R$ is linear. Thus one could say that, up to an additive linear perturbation and a composition with a linear projection onto a subspace of possibly smaller dimension, every convex function on $\R^n$ is  essentially coercive. This decomposition property has been useful in the proofs of several recent results on global smooth approximation and extension by convex functions; see \cite{Azagra2013, AzagraMudarra1, AzagraMudarra2, Azagra2019, AzagraHajlasz}. 

It is natural to wonder whether this decomposition result should still be true of continuous convex functions defined on infinite-dimensional spaces. The purpose of this note is to give an answer to this question.  

Let us begin by setting some notation and definitions.

\begin{defn}\label{defn essential coercivity}
{\em Let $Z$ be a Hilbert space.
We will say that a function $f:Z\to\R$ is {\em essentially coercive} provided that there exists a linear function $\ell:Z\to\R$ such that
$$
\lim_{|z|\to\infty} \left( f(z)-\ell(z)\right)=\infty.
$$

If $X$ is a closed linear subspace of $Z$, we will denote by $P_{X}:Z\to X$ the orthogonal projection, and by $X^{\perp}$ the orthogonal complement of $X$ in $Z$. For a subset $V$ of $Z$, $\textrm{span}(V)$ will stand for the linear subspace spanned by the vectors of $V$, and $\overline{\textrm{span}}(V)$ for the closure of $\textrm{span}(V)$. 

Let us recall that, for a convex function $f:Z\to\R$, the subdifferential of $f$ at a point $x\in Z$ is defined as
$$
\partial f(x)=\{\xi\in Z :  f(y)\geq f(x)+ \langle \xi, y-x\rangle \textrm{ for all } y\in Z\},
$$
and each $\xi\in\partial f(x)$ is called a subgradient of $f$ at $x$. More generally, if $X$ is a Banach space with dual $X^{*}$ and $\langle \cdot, \cdot\rangle$ denotes the duality product, the subdifferential of a continuous convex function $\varphi:X\to\R$ at a point $x\in X$ is defined by
$$
\partial \varphi(x)=\{\xi\in X^{*} :  f(y)\geq f(x)+ \langle \xi, y-x\rangle \textrm{ for all } y\in X\}.
$$
As is well known, for every $x\in X$ the set $\partial \varphi(x)$ is nonempty, closed, convex and bounded. For any unexplained terms of facts in Convex Analysis we refer to the books \cite{BorweinVanderwerffbook, Rockafellar, Zalinescubook}.}
\end{defn}

Now we may give a more precise statement of the decomposition result we mentioned above.

\begin{thm}\label{rigid global behaviour of convex functions}[See the proofs of \cite[Lemma 4.2]{Azagra2013} and \cite[Theorem 1.11]{AzagraMudarra2}]
For every convex function $f:\R^n\to\R$ there exist a unique linear subspace $X_f$ of $\R^n$, a unique vector $v_{f}\in X_{f}^{\perp}$, and a unique essentially coercive convex function $c_{f}:X_f\to\R$ such that $f$ can be written in the form
$$
f(x)=c_{f}(P_{X_{f}}(x)) +\langle v_{f}, x\rangle \textrm{ for all } x\in\R^n.
$$
The subspace $X_{f}$ coincides with $\textrm{span}\{u-w : u\in\partial f(x), w\in\partial f(y), x, y\in\R^n\}$, and the vector $v_f$ coincides with $Q_{X_f}(\xi_0)$ for any $\xi_0\in\partial f(x_0)$, $x_0\in\R^n$, where $Q_{X_f}=I-P_{X_f}$ is the orthogonal projection of $\R^n$ onto $X_{f}^{\perp}$.
\end{thm}
The above characterizations of $X_f$ and $v_f$ do not appear in the statement of \cite[Theorem 1.11]{AzagraMudarra2}, but they are implicit in its proof.

Let us now examine the question as to what extent this result can be generalized for functions $f$ defined on Banach spaces. The first thing we must observe is that, even in the case of a separable Hilbert space, a strict analogue of Theorem \ref{rigid global behaviour of convex functions} is no longer true in infinite dimensions, because there exist continuous convex functions which attain its minima at a single point (in particular their sets of minimizers do not contain half-lines), and yet are not coercive. For instance, in the separable Hilbert space $X=\ell_2$, the function $\varphi(x)=\sum_{n=1}^{\infty}|x_n|^2/2^{n}$ has this property. However, such functions are
{\em directionally coercive}, in the following sense.

\begin{defn}\label{defn essential directional coercivity}
{\em We will say that a function $f$ defined on a Banach space $(X, \|\cdot\| )$ is {\em directionally coercive} provided that for every $x\in X$ and every $v\in X\setminus\{0\}$ we have that
$$
\lim_{t\to\infty}f(x+tv)=\infty.
$$
We will say that $\varphi:X\to\R$ is {\em essentially directionally coercive} provided that there exists $\ell\in X^{*}$ such that the function $\varphi-\ell$ is directionally coercive.}
\end{defn}

It is then natural to wonder whether an analogue of Theorem \ref{rigid global behaviour of convex functions} should hold true at least for separable infinite-dimensional Hilbert spaces if we replace the notion of essential coerciveness with that of essential directional coerciveness. The first main result of this note tells us that this is indeed true.

\begin{thm}\label{rigid global behaviour of convex functions in Hilbert space} Let $Z$ be a separable Hilbert space. For every continuous convex function $f:Z\to\R$ there exist a unique closed linear subspace $X_f$ of $Z$, a unique vector $v_{f}\in X_{f}^{\perp}$, and a unique essentially directionally coercive convex function $c_{f}:X_f\to\R$ such that $f$ can be written in the form
$$
f(z)=c_{f}(P_{X_{f}}(z)) +\langle v_{f}, z\rangle \textrm{ for all } z\in Z.
$$
The subspace $X_{f}$ coincides with $\overline{\textrm{span}}\{u-w : u\in\partial f(z), w\in\partial f(y), z, y\in Z\}$, and the vector $v_f$ coincides with $Q_{X_f}(\xi_0)$ for any $\xi_0\in\partial f(z_0)$, $z_0\in Z$, where $Q_{X_f}=I-P_{X_f}$ is the orthogonal projection of $Z$ onto $X_{f}^{\perp}$.
\end{thm}

If we want to extend this result to the class of arbitrary Banach spaces, we must take into account that on every Banach space $X$ which does not admit any Hilbertian renorming there always exist closed subspaces $Y\subset X$ which are not complemented in $X$, that is, there is no continuous linear projection $P:X\to Y$; see \cite{LindenstraussTzafriri}. This fact precludes a more or less literal generalization of Theorem \ref{rigid global behaviour of convex functions in Hilbert space}. Nonetheless, if we are willing to consider $X_f$ as a quotient space of $Z$ instead of a subspace of $Z$, then we can give a weak version of Theorem \ref{rigid global behaviour of convex functions in Hilbert space} for the class of all Banach spaces.

\begin{thm}\label{rigid global behaviour of convex functions in Banach spaces} Let $Z$ be a  Banach space. For every continuous convex function $f:Z\to\R$ there exists a unique closed linear subspace $Y_f$ of $Z$ such that, for the quotient space $X_f :=Z/Y_{f}$ and the natural projection $\pi=\pi_{f}:Z\to X_f$, the function $f$ can be written in the form
$$
f(z)=c(\pi(z)) +\ell(z) \textrm{ for all } z\in Z,
$$
where $c:X_f\to [a, \infty)$ is a convex function which is not constant on any line, $a\in c(X_f)$, and $\ell\in Z^{*}$.

The subspace $Y_{f}$ coincides with $\{v\in Z: f(z_0+tv)-f(z_0)-\langle\xi_0, tv\rangle=0 \textrm{ for all } t\in\R\}$, where $z_0$ is any point in $Z$ and $\xi_0$ is any linear form in $\partial f(z_0)$. 
\end{thm}

If $X$ is separable we can make this result stronger.

\begin{thm}\label{rigid global behaviour of convex functions in separable Banach spaces} Let $Z$ be a  separable Banach space. For every continuous convex function $f:Z\to\R$ there exists a unique closed linear subspace $Y_f$ of $Z$ such that, for the quotient space $X_f :=Z/Y_{f}$ and the natural projection $\pi=\pi_{f}:Z\to X_f$, the function $f$ can be written in the form
$$
f(z)=c(\pi(z)) +\ell(z) \textrm{ for all } z\in Z,
$$
where $c:X_f\to\R$ is a convex function which is essentially directionally coercive, and $\ell\in Z^{*}$.

The subspace $Y_{f}$ coincides with $\{v\in Z: f(z_0+tv)-f(z_0)-\langle\xi_0, tv\rangle=0 \textrm{ for all } t\in\R\}$, where $z_0$ is any point in $Z$ and $\xi_0$ is any linear form in $\partial f(z_0)$. 
\end{thm}

The following example shows that, even in the Hilbertian setting, if $Z$ is not separable then it is not generally possible to obtain a decomposition $f(z)=c_{f}(\pi(z)) +\ell(z)$ with $c_f$ essentially directionally coercive (as opposed to $c_f$ just not being constant on any line). This means that the statement of Theorem \ref{rigid global behaviour of convex functions in separable Banach spaces} would be generally false if we let $Z$ be a nonseparable Banach space, and consequently that Theorem \ref{rigid global behaviour of convex functions in Banach spaces} is the best result of its kind that one can obtain in the nonseparable setting.

\begin{ex}\label{not true for nonseparable Hilbert spaces}
{\em Let $\Gamma$ be an uncountable set, and set $Z=\ell_2(\Gamma)$, the space of all functions $x:\Gamma\to\R$ such that $\sum_{\gamma\in\Gamma}|x(\gamma)|<\infty$. Recall that this summability condition implies that $\textrm{supp}(x):=\{\gamma\in\Gamma : x(\gamma)\neq 0\}$ is countable for every $x\in\ell_2(\Gamma)$, and that $\ell_2(\Gamma)$ becomes a nonseparable Hilbert space when endowed with the norm
$$
\|x\|=\left(\sum_{\gamma\in\Gamma}|x(\gamma)|^2\right)^{1/2},
$$
and the associated inner product
$$
\langle x, y\rangle=\sum_{\gamma\in\Gamma}x(\gamma) y(\gamma).
$$
As is customary we will also denote $x=(x_{\gamma})_{\gamma\in\Gamma}$, where $x_{\gamma}:=x(\gamma)$ for every $\gamma\in\Gamma$.

We are about to construct a continuous convex function $f:Z\to\R$ such that, when we apply Theorem \ref{rigid global behaviour of convex functions in Banach spaces} with it, we get $X_f=Z$, and yet there exists no linear form $\ell:Z\to\R$ such that $f-\ell$ is directionally coercive. Let us define $\theta:\R\to [0, \infty)$ by
\begin{equation}\label{definition of theta}
\theta(t)=
\begin{cases}
0 & \textrm{ if } t\leq 0 \\
t^2 & \textrm{ if } t\geq 0,
\end{cases}
\end{equation}
and set $f:Z\to\R$,
$$
f(x):=\sum_{\gamma\in\Gamma} \theta(x_{\gamma}).
$$
It is easy to see that $f$ is a continuous convex function such that $0\leq f(x)\leq \|x\|^2$ for every $x\in Z$. We claim that $f$ is not constant on any line $L=\{x+tv : t\in\R\}$, with $x, v\in Z$, $v\neq 0$. Indeed, since $v\neq 0$ there exists some $\beta\in\Gamma$ such that $v_{\beta}\neq 0$. If $v_{\beta}>0$ then, for all $t>-x_{\beta}/v_{\beta}$ we have
$$
f(x+tv)\geq \theta(x_{\beta}+t v_{\beta})=(x_{\beta}+t v_{\beta})^2,$$
and it follows that $\lim_{t\to+\infty}f(x+tv)=+\infty$. Similarly, if $v_{\beta}<0$ then we get $\lim_{t\to-\infty}f(x+tv)=+\infty$. In either case, $f$ cannot be constant on $L$. 

Since $0\in\partial f(0)$, this also shows, thanks to the characterization of $Y_f$ provided by Theorem \ref{rigid global behaviour of convex functions in Banach spaces}, that $Y_f=\{0\}$, which implies that $X_f=Z$ and $\pi_f$ is the identity.

Let us finally see that there exists no linear form $\ell:Z\to\R$ such that $f-\ell$ is directionally coercive. Suppose there exists such $\ell$. Then $\ell(x)=\langle v, x\rangle$ for some $v\in \ell_2(\Gamma)$. Since $\Gamma$ is uncountable and $\textrm{supp}(v)=\{\gamma\in\Gamma : v_{\gamma}\neq 0\}$ is countable, there exists some $\alpha\in\Gamma$ such that $v_{\alpha}=0$. Let $e$ denote the point of $Z$ such that $e_{\alpha}=1$ and $e_{\gamma}=0$ for all $\gamma\neq\alpha$. Because $f-\ell$ is supposed to be directionally coercive we should have $\lim_{t\to\infty}f(-te)-\ell(-te)=\infty$. However, we have
$$
f(-te)-\ell(-te)=f(-te)+t\langle e, v\rangle=f(-te)+tv_{\alpha}=f(-te)=\theta(-t)=0
$$
for all $t>0$. \qed
}
\end{ex}

\bigskip

\section{Proofs of the main results}

Let us now give the proofs of Theorems \ref{rigid global behaviour of convex functions in Hilbert space}, \ref{rigid global behaviour of convex functions in Banach spaces} and \ref{rigid global behaviour of convex functions in separable Banach spaces}.

\begin{proof}[Proof of Theorem \ref{rigid global behaviour of convex functions in Banach spaces}] Let $Z$ be a Banach space, and $f:Z\to\R$ a continuous convex function. Let us define, for every point $z_0\in Z$ and every $\xi_0\in\partial f(z_0)$, the set
$$
Y(f, z_0, \xi_0)=\{v\in Z : f(z_0+tv)-f(z_0)-\langle\xi_0, tv\rangle=0 \textrm{ for all } t\in\R\}.
$$
\begin{lem}\label{Yf is a subspace not depending on z and xi}
For every $z_1, z_2\in Z$, $\xi_1\in\partial f(z_1), \xi_2\in\partial f(z_2)$, we have the following:
\begin{enumerate}
\item  $Y(f, z_1, \xi_1)$ is a closed linear subspace of $Z$;
\item $Y(f, z_1, \xi_1)=Y(f, z_2, \xi_2)$ (that is, $Y(f, z, \xi)$ does not depend on $z, \xi$);
\item $f(z_1+v)=f(z_1)+\langle \xi_1, v\rangle$ for all $v\in Y(f, z_1, \xi_1)$;
\item $\langle \xi_2-\xi_1, v\rangle=0$ for all $v\in Y(f, z_1, \xi_1)$.
\end{enumerate} 
\end{lem}
\begin{proof}
Let us denote $Y= Y(f, z_0, \xi_0)$ for convenience. From the definition it is clear that $\lambda v\in Y$ for all $v\in Y$, $\lambda\in\R$. So, in order to check that $Y$ is a subspace it is enough to see that, given $v_1, v_2\in Y$, we have that $v_1+v_2\in Y$. We may write 
$
t(v_1+v_2)=\frac{1}{2} t 2v_1+\frac{1}{2}t 2v_2
$, and since $f$ is convex, $\xi_0\in\partial f(z_0)$ and $2v_1, 2v_2\in Y$, we have
\begin{eqnarray*}
& & 0\leq f(z_0+t(v_1+v_2))-f(z_0)-\langle \xi_0, t(v_1+v_2)\rangle \\
& & \leq \frac{1}{2}\left( f(z_0+t2v_1)-f(z_0)-\langle\xi_0, t2v_1\rangle\right)
+\frac{1}{2}\left( f(z_0+t2v_2)-f(z_0)-\langle\xi_0, t2v_2\rangle\right) \\
& & \leq 0+0=0,
\end{eqnarray*}
hence $f(z_0+t(v_1+v_2))-f(z_0)-\langle \xi_0, t(v_1+v_2)\rangle=0$ for all $t\in\R$, that is, $v_1+v_2\in Y$. This shows $(1)$ (the fact that $Y$ is closed is obvious by continuity of $f$ and $\xi_0$).

In order to see that $Y(f, z, \xi)$ does not depend on $z, \xi$, fix $z_1, z_2\in Z$, $\xi_1\in\partial f(z_1)$, $\xi_2\in\partial f(z_2)$, and let us check that  $Y(f, z_1, \xi_1)=Y(f, z_2, \xi_2)$. Up to an affine perturbation and a translation of variables, we may assume that $z_2=0, \xi_2=0$, and $f(0)=0$. So we have $f(tv)=0$ for all $v\in Y(f, 0, 0)$, $t\in\R$. Let us fix $v\in Y(f, 0, 0)$, and let us see that $f$ is constant on the line $L(z_1, v)=\{z_1+tv : t\in\R\}$. Suppose not: then either $\lim_{t\to\infty}f(z_1+tv)=\infty$ or $\lim_{t\to-\infty}f(z_1+tv)=\infty$. Assume for instance that $\lim_{t\to\infty}f(z_1+tv)=\infty$, and take $t_2>t_1>0$ large enough so that
$$
f(z_1+t_2v)>f(z_1+t_1v);
$$
then by convexity, for any $\eta_2\in\partial f(z_1+t_2v)$, we have that 
$$
\langle \eta_2, (t_2-t_1)v\rangle\geq f(z_1+t_2 v)-f(z_1+t_1 v)>0,
$$
which implies that 
$$
\langle \eta_2, v\rangle>0.
$$
Then, using again the convexity of $f$, we have, for $t>t_2$, 
$$
0=f(tv)\geq f(z_1+t_2 v)+\langle \eta_2, tv-z_1-t_2v\rangle=
f(z_1+t_2v)+\langle\eta_2, tv\rangle-\langle\eta_2, z_1 +t_2v\rangle\to\infty
$$
as $t\to\infty$, which is absurd. Therefore $f$ is constant on $L(z_1, v)$. This implies that $f(z_1+tv)-f(z_1)-\langle\xi_1, tv\rangle=0$ for all $t\in\R$. Thus $v\in Y(f, z_1, \xi_1)$. This argument shows that $Y(f, z_2,\xi_2)\subseteq Y(f, z_1, \xi_1)$. Similarly one can check that $Y(f, z_1, \xi_1)\subseteq Y(f, z_2, \xi_2)$. This proves $(2)$.

Property $(3)$ is obvious from the definition of $Y(f, z, \xi)$. To check $(4)$, note that since $\xi_2\in\partial f(z_2)$ we have 
$$
f(z_1+tv)\geq f(z_2)+\langle\xi_2, z_1+tv-z_2\rangle.
$$
By combining this with $(3)$ we get, for all $v\in Y(f, z_1, \xi_1)$,  that
$$
f(z_1)+\langle\xi_1, tv\rangle=f(z_1+tv)\geq f(z_2)+\langle\xi_2, z_1+tv-z_2\rangle,
$$
hence
$$
f(z_1)+\langle\xi_1-\xi_2, tv\rangle \geq f(z_2)+\langle\xi_2, z_1-z_2\rangle
$$
for all $t\in\R$, $v\in Y(f, z_1, \xi_1)$. If $\langle \xi_1-\xi_2, v\rangle >0$ this implies that
$$
-\infty=\lim_{t\to-\infty}f(z_1)+\langle\xi_1-\xi_2, tv\rangle\geq f(z_2)+\langle\xi_2, z_1-z_2\rangle,
$$
which is absurd. Similarly, if $\langle \xi_1-\xi_2, v\rangle <0$ then we get
$$
-\infty=\lim_{t\to+\infty}f(z_1)+\langle\xi_1-\xi_2, tv\rangle\geq f(z_2)+\langle\xi_2, z_1-z_2\rangle.
$$
Therefore we must have $\langle \xi_2-\xi_1, v\rangle =0$ for all $v\in Y(f, z_1, \xi_1)$.
\end{proof}

From now on we will denote by $Y_f$ any (hence all) of the subspaces $Y(f, z, \xi)$  appearing in the statement of Lemma  \ref{Yf is a subspace not depending on z and xi}.
Fix $z_0\in Z$, and $\xi_0\in\partial f(z_0)$, set
$$
\ell :=\xi_0,
$$
and define $\varphi:Z\to\R$ by
$$
\varphi(z):=f(z)-\ell(z).
$$
\begin{lem}
The function $\varphi$ satisfies $\varphi(z)=\varphi(z')$ for all $z, z'\in Z$ with $z-z'\in Y_f$. Hence the function
$$
\widehat{\varphi}: X_f:=Z/Y_f\to\R,  \,\,\, \,\,\, \widehat{\varphi}(z+Y_f):=\varphi(z)
$$
is well defined, and satisfies
$$
\varphi=\widehat{\varphi}\circ \pi, 
$$
where $\pi:Z\to Z/Y_f=X_f$ is the natural projection.
\end{lem} 
\begin{proof}
For all $z, z'\in Z$ with $z-z'\in Y_f$ we have
\begin{eqnarray*}
& & \varphi(z)-\varphi(z')=f(z)-f(z')-\langle \xi_0, z-z'\rangle \\
& & =f(z')+\langle \xi', z-z'\rangle-f(z')-\langle \xi_0, z-z'\rangle \\
& & =\langle \xi'-\xi_0, z-z'\rangle=0,
\end{eqnarray*}
where in the second equality we use Lemma \ref{Yf is a subspace not depending on z and xi}(3) with $z_1=z'$ and any $\xi'\in\partial f(z')$, and in the last equality we use Lemma \ref{Yf is a subspace not depending on z and xi}(4) (with  $\xi_2=\xi'$ and $\xi_1=\xi_0$).
\end{proof}

\begin{lem}
The function $\widehat{\varphi}:X_f\to\R$ is convex and continuous.
\end{lem}
\begin{proof}
Let us first check that $\widehat{\varphi}$ is convex. Given $z, z'\in Z$ and $t\in [0,1]$ we have, since $2Y_f=Y_f$ and $\varphi$ is convex, that
\begin{eqnarray*}
& & \widehat{\varphi}\left( t(z+Y_f)+(1-t)(z'+Y_f) \right)=
\widehat{\varphi}\left( tz+(1-t)z'+Y_f\right)=\varphi(tz+(1-t)z')\\
& & \leq t \varphi(z)+(1-t)\varphi(z')=t \widehat{\varphi}(z+Y_f)+(1-t)\widehat{\varphi}(z'+Y_f).
\end{eqnarray*}
Therefore $\widehat{\varphi}$ is convex. Then, to prove that $\widehat{\varphi}$ is continuous, it is enough to see that $\widehat{\varphi}$ is locally bounded. Given $z_0+Y_f\in X_f$, because $\varphi$ is locally bounded, there exists some $r_0>0$ such that $\varphi(z_0 +r_0 B_Z)$ is bounded. But, by Lemma \ref{Yf is a subspace not depending on z and xi}
$$
\varphi(z_0+ r_0 B_Z)=\varphi(z_0+Y_f+r_0 B_Z)=\varphi(z_0+Y_f-Y_f+r_0 B_Z)=\widehat{\varphi}(z_0+Y_f+N),
$$
where $N :=Y_f+r_0 B_Z$ is a neighborhood of $0$ in $X_f=Z/Y_f$, so it follows that $\widehat{\varphi}$ is bounded on the neighborhood $z_0+Y_f+N$ of $z_0+Y_f$ in $X_f$.
\end{proof}
To finish the proof of Theorem \ref{rigid global behaviour of convex functions in Banach spaces} let us set $$c:=\widehat{\varphi}.$$ 
Since $\ell=\xi_0\in\partial f(z_0)$, we have that 
$$\varphi(z)=f(z)-\ell(z)\geq f(z_0)-\ell(z_0)=\varphi(z_0)=:a,$$ 
hence $\widehat{\varphi}\circ\pi \geq a$, which implies that $c\geq a\in c(X_f)$. It only remains for us to check that $\widehat{\varphi}$ is not constant on any line. Suppose, for the sake of contradiction, that there exist $\widehat{z}_0=z_0+Y_f, \widehat{v_0}=v_0+Y_f\in X_f$ such that $\widehat{v_0}\neq 0$ (that is, $v_0\notin Y_f$), and $\widehat{\varphi}( \widehat{z}_0+t\widehat{v}_0)=\widehat{\varphi}(\widehat{z}_0)$ for all $t\in\R$. This means that
$$
\varphi(z_0)=\widehat{\varphi}(z_0+Y_f)=\widehat{\varphi}(z_0+Y_f+t(v_0+Y_f))=
\widehat{\varphi}(z_0+tv_0+Y_f)=\varphi(z_0+tv_0),
$$
which since $\varphi=f-\xi_0$ implies $f(z_0+tv_0)=f(z_0)+\langle \xi_0, tv_0\rangle$ for all $t\in\R$, that is, $v_0\in Y(f, x_0, \xi_0)=Y_f$, contrary to the assumption.
\end{proof}

\bigskip

\begin{proof}[Proof of Theorem \ref{rigid global behaviour of convex functions in separable Banach spaces}] By Theorem \ref{rigid global behaviour of convex functions in Banach spaces} we already know that, given a continuous convex function $f:Z\to\R$ there exists a unique closed linear subspace $Y_f$ of $Z$ such that, for $X_f :=Z/Y_{f}$ and the natural projection $\pi=\pi_{f}:Z\to X_f$, we have a decomposition
$$
f(z)=c(\pi(z)) +\ell(z) \textrm{ for all } z\in Z,
$$
where $\ell\in Z^{*}$, $c:X_f\to [a, \infty)$ is not constant on any line, and $a$ is the (attained) infimum of the set $c(X_f)$. Thus it will suffice to check that $c$ is essentially directionally coercive. This is a consequence of the following.

\begin{prop}\label{for separable not constant on any line implies ed coercive}
Let $X$ be a separable Banach space, and $\varphi: X\to [a, \infty)$ be a continuous convex function which is not constant on any line. Then $\varphi$ is essentially directionally coercive.
\end{prop}
\begin{proof}
Define $\psi(x)=\varphi(x)-\varphi(0)-\langle \xi_0, x\rangle$, where $\xi_0\in\partial\varphi(0)$. We have that $\psi(0)=0\leq\psi(x)$ for all $x\in X$, $\psi$ is convex and continuous, and $\psi$ is not constant on any line (indeed, if for some $x_0, v_0\in X$ we had $\psi(x_0+tv_0)=\psi(x_0)$ for all $t\in\R$ then we would get $\varphi(x_0+tv_0)=\varphi(x_0)+t\langle \xi_0, v_0\rangle$ for all $t\in\R$. If $\langle \xi_0, v_0\rangle\neq 0$ this equality contradicts the assumption that $\varphi\geq a$. And if $\langle \xi_0, v_0\rangle=0$ it contradicts the assumption that $\varphi$ is not constant on any line). 

Next consider the closed convex set
$$
C :=\{x\in X : \psi(x)\leq 1\}.
$$
For every $x\in \partial C$ we choose $\xi_x\in\partial \psi(x)$, so we can write
$$
C=\bigcap_{x\in\partial C}H_{x}^{-}, \textrm{ where } H_{x}^{-}:=\{y\in X : \langle \xi_x, y-x\rangle \leq 0\},
$$
or equivalently
$$
X\setminus C=\bigcup_{x\in\partial C}H_{x}^{+}, \textrm{ where } H_{x}^{+}:=\{y\in X : \langle \xi_x, y-x\rangle > 0\},
$$
and since $X$ is separable there exists a sequence $\{x_n\}_{n\in\N}\subset\partial C$ such that, writing $\xi_n=\xi_{x_n}$, we have
\begin{eqnarray}\label{complement of C as a countable union of halfspaces}
X\setminus C=\bigcup_{n=1}^{\infty}\{y\in X: \langle \xi_n, y-x_n\rangle >0\}.
\end{eqnarray}
\begin{lem}\label{the xin separate points}
The set $\{\xi_n : n\in\N\}\subset X^{*}$ separates points of $X$ (that is to say, for every $v\neq 0$ there exists $n=n(v)$ such that $\xi_n (v)\neq 0$).
\end{lem}
\begin{proof}
Fix $v\in X\setminus\{0\}$. Since $\psi$ is convex and is not constant on the line $\{tv : t\in\R\}$, we have $\lim_{t\to\infty}\psi(tv)=\infty$ or $\lim_{t\to -\infty}\psi(tv)=\infty$. Up to replacing $v$ with $-v$ (and noting that  $\xi_n(v)\neq 0$ if and only if $\xi_n(-v)\neq 0$) we may assume without loss of generality that $\lim_{t\to\infty}\psi(tv)=\infty$.
Then we can find $s_v>0$ such that $\psi(s_v v)>2$, which means that $s_v v\in X\setminus C$, and according to \eqref{complement of C as a countable union of halfspaces} there exists $n=n(v)$ such that 
$
\langle \xi_n, s_v v-x_n\rangle >0.
$
On the other hand, since $0\in C$ we also have
$
\langle\xi_n, -x_n\rangle\leq 0.
$
Therefore the function $t\mapsto \langle \xi_n, t v-x_n\rangle$ is not constant, which implies that $t\mapsto \langle\xi_n, tv\rangle$ is not constant either, and this means that $\langle\xi_n, v\rangle\neq 0$.
\end{proof}

Also note that we have
\begin{eqnarray}\label{psi bounded below by a coercive sup}
\psi(x)\geq \sup_{n\in\N}\max\left\{0, \, \psi(x_n)+\langle\xi_n, x-x_n\rangle\right\}
\end{eqnarray}
for all $x\in X$. Now let us define
$$
\xi=\sum_{n=1}^{\infty}\frac{1}{2^{n+1} \max\{1, \|\xi_n\|\}}\xi_n \in X^{*}.
$$
The proof of Proposition \ref{for separable not constant on any line implies ed coercive} will be complete as soon as we check that
$\psi-\xi$ is directionally coercive. To this end, for each $x\in X$ and $v\in X\setminus\{0\}$, let us see that the function
$$
h(t):=\psi(x+tv)-\langle\xi, x+tv\rangle
$$
satisfies $\lim_{t\to\infty}h(t)=\infty$. We will distinguish three cases.

{\em Case 1.} Suppose that $\langle \xi, v\rangle <0$. Then, since $\psi\geq\psi(0)=0$, we have
$
h(t)\geq -t\langle\xi, v\rangle -\langle\xi, x\rangle\to\infty
$
as $t\to\infty$.

{\em Case 2.} Suppose that $\langle \xi, v\rangle >0$. Then there must exist some $n\in\N$ with $\langle\xi_n, v\rangle >0$, and in particular
$$
\alpha:=\sup_{n\in\N}\, \frac{\langle\xi_n, v\rangle}{\max\{1, \|\xi_n\|\}}>0.
$$
Notice that, for all $t>0$,
\begin{eqnarray}\label{t halves bound}
t\langle \xi, x\rangle=\sum_{n=1}^{\infty}\frac{ t\langle \xi_n, x\rangle}{2^{n+1}\max\{1, \|\xi_n\|\}}\leq \frac{t}{2}\alpha,
\end{eqnarray}
and fix some $k\in\N$ such that 
\begin{eqnarray}\label{two thirds of alpha bound}
\frac{\langle\xi_k, v\rangle}{\max\{1, \|\xi_k\|\}}>\frac{2}{3}\alpha >0.
\end{eqnarray}
By combining \eqref{t halves bound} and \eqref{two thirds of alpha bound} we then have, for all $t>0$, that
$$
t\langle \xi_k, v\rangle >\frac{2}{3}  \alpha t \max\{1, \|\xi_k\|\} \geq \frac{2}{3} \alpha t \geq \frac{4}{3} t\langle \xi, x\rangle,
$$
and therefore
$$
-t\langle \xi, x\rangle \geq -\frac{3}{4} t \langle \xi_k, v\rangle,
$$
which implies
\begin{eqnarray*}
& & h(t)=\psi(x+tv)-t\langle\xi, v\rangle -\langle \xi, x\rangle\\
& & \geq \psi(x_k)+\langle\xi_k, x+tv -x_k\rangle -\frac{3}{4}t \langle \xi_k, v\rangle -\langle \xi, x\rangle \\
& & = \psi(x_k)+\langle\xi_k, x-x_k\rangle+\frac{1}{4} t\langle\xi_k, v\rangle -\langle \xi, x\rangle \to\infty
\end{eqnarray*}
as $t\to\infty$.

{\em Case 3.}  Suppose finally that $\langle \xi, v\rangle =0$. Since $v\neq 0$, Lemma \ref{the xin separate points} gives us some $k$ with $\langle\xi_k, v\rangle\neq 0$. Then for $\langle \xi, v\rangle =0$ to be true there have to be distinct numbers $n, m\in\N$ so that $\langle\xi_n, v\rangle>0$ and $\langle\xi_m, v\rangle<0$. In particular
$$
\sup_{j\in\N}\, \frac{\langle\xi_j, v\rangle}{\max\{1, \|\xi_j\|\}}>0,
$$
and therefore the same argument as in Case 2 applies.
\end{proof}
The proof of Theorem \ref{rigid global behaviour of convex functions in separable Banach spaces} is now complete.
\end{proof}

\bigskip

\begin{proof}[Proof Theorem \ref{rigid global behaviour of convex functions in Hilbert space}] By the proof of Theorem \ref{rigid global behaviour of convex functions in separable Banach spaces} we know that there exists a unique closed linear subspace $Y_f$ of $Z$ such that, for the quotient space $X_f :=Z/Y_{f}$ and the natural projection $\pi=\pi_{f}:Z\to X_f$, the function $f$ can be written in the form
$$
f(z)=c(\pi(z)) +\langle \xi, z \rangle \textrm{ for all } z\in Z,
$$
where $c:X_f\to\R$ is a convex function which is essentially directionally coercive, more precisely, $c(z+Y_f)=\widehat{\varphi}(z+Y_f) :=\varphi(z):=f(z)-\langle\xi, z\rangle$, and $\xi\in\partial f(z_0)$ (for some $z_0\in Z$). Since $Z$ is a Hilbert space, the quotient space $X_f=Z/Y_f$ is isometric to $Y_{f}^{\perp}$, and with this identification the natural projection $\pi:Z\to X_f$ becomes the orthogonal projection $P_{X_{f}}:Z\to X_{f}$. If we consider the orthogonal decomposition $Z=X_{f}\oplus Y_{f}$ and write $z=(x,y)$ with $x\in X_{f}$, $y\in Y_f$ and similarly we identify $Z^{*}$ with $Z=X_{f}\oplus Y_{f}$ and write $\xi=(\xi_X, \xi_Y)$ then we have that
$$
f(x,y)=\langle \xi_X, x\rangle +\langle\xi_Y, y\rangle +\widehat{\varphi}(x),
$$
where $\widehat{\varphi}:X_f\to\R$ is essentially directionally coercive. Then it is easy to check that $\eta_Y=\xi_Y$ for every $\eta=(\eta_X, \eta_Y)\in\partial f(z)$ and every $z\in Z$. Moreover,
$$
f(x,y)=\widetilde{c_f}(x)+\langle\xi_Y, y\rangle,
$$
where $$\widetilde{c_f}(x):=\langle \xi_X, x\rangle +\widehat{\varphi}(x)=f(x,0)$$ is essentially directionally coercive. Theorem \ref{rigid global behaviour of convex functions in Hilbert space} will follow immediately from these observations as soon as we check that 
$$
X_{f}=\overline{\textrm{span}}\{u-w : u\in\partial f(z), w\in\partial f(y), z, y\in Z\}.
$$ 
To this end, let $W$ stand for the closed subspace on the right side. If $\xi_j\in\partial f(z_j)$, $j=1, 2$, then by Lemma \ref{Yf is a subspace not depending on z and xi}(4) we have $\xi_1-\xi_2\in Y_{f}^{\perp}=X_f$, and it follows that $W\subseteq X_f$. Now, if we did not have $W=X_f$ then there would exist $u\in X_f\setminus\{0\}$ such that $u\perp W$. For each $t\in\R$ choose $\eta_t\in \partial f(tu)$. Then $\eta_t-\eta_0\in W$, and
$$
0\leq f(tu)-f(0)-\langle \eta_0, tu\rangle \leq\langle \eta_t-\eta_0, tu\rangle=0
$$
for all $t\in\R$,
which implies that $u\in Y_f$, a contradiction. Thus we must have $W=X_f$.
\end{proof}

\medskip

\section{Further remarks}

\begin{rem}
{\em If $Z=\R^n$ then it is not difficult to show that every directionally coercive convex function is coercive. Therefore Theorem \ref{rigid global behaviour of convex functions} can be obtained as a corollary to Theorem \ref{rigid global behaviour of convex functions in Hilbert space}.}
\end{rem}

One can also restate Theorem \ref{rigid global behaviour of convex functions} in the following equivalent form.
\begin{thm}\label{rigid global behaviour of convex functions version 2}
For every convex function $f:\R^n\to\R$ there exist a unique linear subspace $X_f$ of $\R^n$ such that $f$ can be written in the form
$$
f(x)=\varphi(P_{X_{f}}(x)) +\ell(x) \textrm{ for all } x\in\R^n,
$$
where $\ell:\R^n\to\R$ is linear and $\varphi:X_f\to\R$ is a convex function which attains a strict minimum.
\end{thm}
Indeed, from Theorem \ref{rigid global behaviour of convex functions} it is obvious that this statement is true if we only ask that $\varphi$ be essentially coercive, that is, there exist linear forms $\ell:\R^n\to\R$ and $\xi:X_f\to\R$, and a function $\varphi:X_f\to\R$ such that $f=\varphi\circ P_{X_f}+\ell$, and $\varphi-\xi$ is coercive. But then by a result due to Moreau and Rockafellar (see \cite[Corollary 4.4.11]{BorweinVanderwerffbook}) the Fenchel conjugate $\varphi^{*}$ of $\varphi$ is continuous at $\xi$, and $\xi$ is in the interior of the domain $D$ of $\varphi^{*}$. Since the convex function $\varphi^{*}: \textrm{int}(D)\subseteq X_{f}\to\R$ is differentiable almost everywhere we may find $\xi_0\in \textrm{int}(D)$ such that $\varphi^{*}$ is differentiable at $\xi_0$, and therefore, according to \cite[Theorem 5.2.3]{BorweinVanderwerffbook}, $\varphi$ is strongly exposed at $x_0:=\nabla \varphi^{*}(\xi_0)$, which means that $\varphi-\xi_0$ attains a strong minimum at $x_0$. Thus, we can write
$$
f=(\varphi-\xi_0)\circ P_{X_f}+\xi_0\circ P_{X_{f}}+\ell,
$$
that is, 
$$
f=\widetilde{\varphi}\circ P_{X_f}+\widetilde{\ell},
$$
where $\widetilde{\varphi}:=\varphi-\xi_0:X_{f}\to\R$ attains a strict maximum and $\widetilde{\ell}:=\xi_0\circ P_{X_{f}}+\ell:\R^n\to\R$ is linear. This proves Theorem \ref{rigid global behaviour of convex functions version 2}. Conversely, it is not difficult to show that every convex function $\psi:\R^n\to\R$ which attains a strict minimum must be coercive, from which one can easily see that Theorem \ref{rigid global behaviour of convex functions version 2} also implies Theorem \ref{rigid global behaviour of convex functions}. 

It is then natural to wonder whether a similar result should hold true if we replace $\R^n$ with a separable Hilbert space. That is to say, is the statement of Theorem \ref{rigid global behaviour of convex functions in Hilbert space} true if we ask that $c_f:X_f\to\R$ be a convex function which attains a strict minimum at some point? The answer is negative, as we next show.

\begin{ex}
{\em Let $\theta:\R\to [0, \infty)$ be the function defined by \eqref{definition of theta} in Example \ref{not true for nonseparable Hilbert spaces}, let $Z=\ell_{2}$, denote by $\{e_n\}_{n\in\N}$ and $\{e_n^{*}\}_{n\in\N}$ the canonical bases of $Z$ and $Z^{*}$, and define $f:Z\to [0, \infty)$ by
$$
f(x)=\sum_{n=1}^{\infty}\frac{1}{2^{n}}\theta\left( e_{n}^{*}(x)-n\right).
$$
It is clear that $f$ is convex and continuous. Let us see that $f$ is not constant on any line. For each $x, v\in Z$ with $v\neq 0$ we have $e_{n}^{*}(v)\neq 0$ for some $n$. Suppose for instance $e_{n}^{*}(v)>0$. Then
$$
f(x+tv)\geq \theta( e_{n}^{*}(x)+t e_{n}^{*}(v)-n)\to\infty
$$
as $t\to\infty$. Similarly, if $e_{n}^{*}(v)<0$ we see that $\lim_{t\to-\infty}f(x+tv)=\infty$. Therefore, with the notation of Theorem \ref{rigid global behaviour of convex functions in Hilbert space} we have $X_{f}=Z$, and one can find a continuous linear form $\xi$ such that $f-\xi$ is essentially directionally coercive. However, let us show that there is no continuous linear form $\xi:Z\to\R$ such that $f-\xi$ attains a strict minimum. For the sake of contradiction, suppose $f-\xi$ had a strict minimum at $x$, take $m=m_x$ large enough so that  $m>e_{m}^{*}(x)$, and set $v=e_m$. Then we would have $\theta\left( e_{m}^{*}(x)+t-m\right)=\theta\left( e_{m}^{*}(x)-m\right)=0$ for all $|t|< m-e_{m}^{*}(x)$, hence
$$
f(x+tv)=\frac{1}{2^m}\theta\left( e_{m}^{*}(x)+t-m\right)+ \sum_{n\neq m}\frac{1}{2^n} \theta\left( e_{n}^{*}(x)-n\right)=\sum_{n=1}^{\infty}\frac{1}{2^{n}}\theta\left( e_{n}^{*}(x)-n\right)=f(x),
$$
and also
$$
(f-\xi)(x+tv)=(f-\xi)(x)-t\xi(v) \,\,\, \textrm{ for } \,\,\, |t|< m-e_{m}^{*}(x),
$$
which is impossible if $f-\xi$ attains a strict minimum at $x$. 
}
\end{ex}

\medskip


\end{document}